\documentclass{birkjour}
\usepackage{subcaption}
\usepackage{graphicx}
\usepackage{color}
\usepackage{pinlabel}
 \newtheorem{theorem}{Theorem}[section]
 \newtheorem{corollary}[theorem]{Corollary}

 \theoremstyle{definition}
 
 \theoremstyle{remark}

 \numberwithin{equation}{section}

\newcommand{\rpt}{\ensuremath{\mathbb{R}\mathrm{P}^2}}
\newcommand{\rp}{\ensuremath{\mathbb{R}\mathrm{P}^2} }
\newcommand{\I}{\mathcal{I}}
\newcommand{\ba}{\setminus}

\begin{document}

%-------------------------------------------------------------------------
% editorial commands: to be inserted by the editorial office
%
%\firstpage{1} \volume{228} \Copyrightyear{2004} \DOI{003-0001}
%\seriesextraline{This is the Concrete Title of this Book\br H.E. R and S.T.C. W, Eds.}
%\firstpage{1}
%\linenumbers

\title[Ribbon graph minors and low-genus partial duals]
 {Ribbon graph minors and low-genus partial duals}

%----------Author 1
\author[I.~Moffatt]{Iain Moffatt}
\address{Department of Mathematics,
Royal Holloway,
University of London,
Egham,
Surrey,
TW20 0EX,
United Kingdom.}
\email{iain.moffatt@rhul.ac.uk}
\subjclass{Primary 05C83, 05C10}

\keywords{embedded graph, excluded minor, real projective plane, knots, minor, ribbon graph.}

\date{\today}

\begin{abstract}
We give an excluded minor characterisation of the class of ribbon graphs that admit partial duals of Euler genus at most one.
\end{abstract}

%%% ----------------------------------------------------------------------
\maketitle
%%% ----------------------------------------------------------------------
%\tableofcontents

\section{Statement of results}
Here we are interested in minors of ribbon graphs. For ribbon graph minors it is necessary to contract loops and doing so can create additional vertices and components \cite{MR1906909,MR2507944}. Thus there is a fundamental difference between graph and ribbon graph minors.
While it is known that every minor-closed family of graphs can be characterised by a finite set of excluded minors \cite{MR2099147}, the corresponding result for ribbon graphs and their minors is currently only a conjecture \cite{Moffatt:2013ty}. 
In this note we provide some support for this conjecture by giving an explicit list of excluded minors for an interesting class of ribbon graphs: those that admit partial duals of low  genus. Our main result is the following.
\begin{theorem}\label{t1} 
Let $X_1$--$X_3$ be the ribbon graphs of Figure~\ref{f1}. 
Then a ribbon graph has a partial dual of Euler genus at most one if and only if it has no ribbon graph minor equivalent to $X_1$, $X_2$, or $X_3$. 
\end{theorem}
Although we assume a familiarity with ribbon graphs, ribbon graph minors and partial duals (see, for example,  \cite{MR3086663} for a leisurely overview of these topics, or Sections~2 and~3.2 of  \cite{Moffatt:2013ty} for a brief review), we briefly recall that  the partial dual $G^A$ of a ribbon graph $G$ is the ribbon graph obtained by, roughly speaking, forming the geometric dual of $G$ but only with respect to a given set $A$ of edges of $G$. Partial duality was defined by Chmutov in \cite{MR2507944}. It appears to be a funtamental construction, arising  as a natural operation in knot theory, topological graph theory, graph polynomials, matroid theory, and quantum field theory.  The \emph{Euler genus} of a connected ribbon graph $G$ is its genus if it is non-orientable, and twice its genus otherwise.  For non-connected ribbon graphs it is defined as the sum of the Euler genera of its components.
A connected ribbon graph of Euler genus 0 is a  \emph{plane ribbon graph}, and of Euler genus 1 is an  \emph{\rp ribbon graph}.

Our approach to Theorem~\ref{t1} is as follows. We use a compatibility between partial duals and minors (Equation~\ref{e1}) to reduce the proof of the theorem to showing that a  bouquet (i.e., a 1-vertex ribbon graph) 
whose partial duals are all of Euler genus at least two  contains $X_1$ or $X_2$ as a minor. We then apply a `rough structure theorem' (Theorem~\ref{t2}) that tells us that such a ribbon graph $G$ does not have a particular type of decomposition into a set of plane graphs and one \rp graph. $X_1$ is an obstruction to finding an appropriate   \rp graph ribbon subgraph in $G$ for the decomposition, $X_2$ is an obstruction to finding an appropriate set of plane ribbon subgraphs for the decomposition, and, somewhat surprisingly, $X_1$ is also an obstruction to  plane and \rp subgraphs fitting together in a way required by the rough structure theorem to allow for a low-genus partial dual.

\begin{figure}[t]
\begin{subfigure}[b]{.3\linewidth}
\centering
 \includegraphics[scale=0.4]{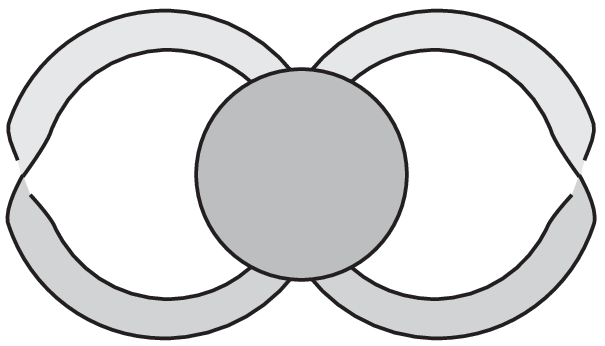}
\caption{$X_1$.}
\label{f1a}
\end{subfigure}
\begin{subfigure}[b]{.3\linewidth}
\centering
 \includegraphics[scale=0.4]{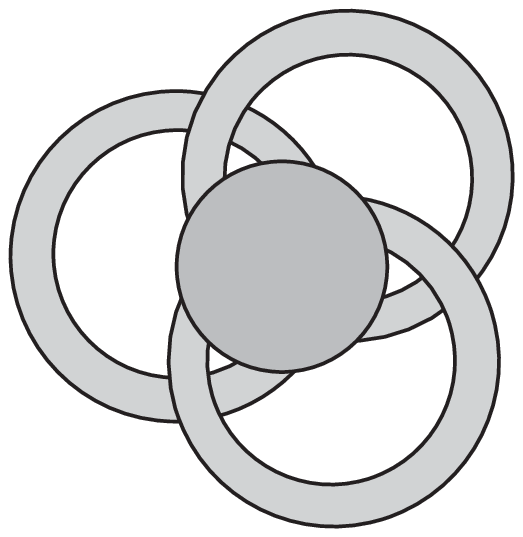}
\caption{$X_2$.}
\label{f1b}
\end{subfigure}
\begin{subfigure}[b]{.3\linewidth}
\centering
 \includegraphics[scale=0.4]{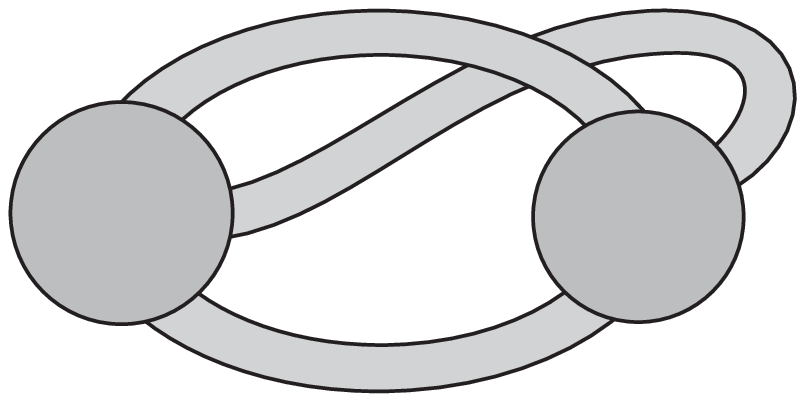}
\caption{$X_3$.}
\label{f1c}
\end{subfigure}
\caption{The excluded minors.}
\label{f1}
\end{figure}

\section{Proof of the main theorem}
The following constructions and results will be needed for the  proof  of Theorem~\ref{t1}.
A vertex $v$ of a ribbon graph $G$ is  a {\em separating vertex} if there are non-trivial ribbon subgraphs $G_1$ and $G_2$ of $G$ such that $G=G_1\cup G_2$ and $G_1\cap G_2=\{v\}$. Let $G$ be a ribbon graph and $A\subseteq E(G)$. We use  $G|_A$ to denote  the restriction of $G$ to $A$ (i.e., the ribbon subgraph with edge set $A$ and  vertices incident to edges in $A$), and $A^c$ to denote the complement of $A$ in $E(G)$. 
Following  \cite{MR2994404},  we say that  $A$ {\em defines a biseparation} of $G$ if every vertex of $G$ that is in both $G|_A$ and $G|_{A^c}$ is a separating vertex of $G$.
If, in addition, every component of $G|_A$ and of $G|_{A^c}$ is plane then we say $A$ defines a \emph{plane-biseparation}; and if exactly one component of   $G|_A$ or of $G|_{A^c}$ is \rp and all of the other components are plane then $A$ defines an \emph{\rpt-biseparation}.

We will use the following rough structure theorem for ribbon graphs with low genus partial duals from \cite{MR2994404} (the plane case first appeared in \cite{MR2928906}).
\begin{theorem}\label{t2}
Let $G$ be a connected ribbon graph and $A\subseteq E(G)$. Then  $G^A$ is a plane ribbon graph (respectively,  \rp  ribbon graph)  if and only if $A$ defines a plane-biseparation (respectively,  \rpt-biseparation) of $G$.
\end{theorem}

We will also need the following result from \cite{Moffatt:2013ty} that says the partial duals of minors are the minors of a partial dual.  
For a ribbon graph $G$ and $A\subseteq E(G)$, we have 
\begin{equation}\label{e1} 
\{ J^A \mid J \text{ is a minor of } G \}  =  \{ H \mid H \text{ is a minor of } G^A \} .
\end{equation}
Here we use the convention that if $H$ is a minor of $G$, and $A\subseteq E(G)$, then by $H^A$ we mean $H^{A\cap E(H)}$.

A  \emph{bouquet} is a ribbon graph that has exactly one vertex. Two edges $e$ and $f$ in a bouquet are \emph{interlaced} if their ends are met in the cyclic order $e\, f\, e\; f$ when travelling round the boundary of the unique vertex of $G$. The \emph{intersection graph} $\I(G)$ of a bouquet $G$ is the vertex weighted simple graph whose vertices set is $E(G)$ and  which two vertices $e$ and $f$ of $\I(G)$ are adjacent if and only if the edges $e$ and $f$ are interlaced in $G$. A vertex  $e$ of $\I(G)$ has weight ``+'' if $e$ is an orientable loop in $G$, and has weight $``-$'' if it is non-orientable.

\begin{proof}[Proof of Theorem~\ref{t1}]
For one implication observe that $X_1$--$X_3$ do not have partial duals of Euler genus less than 1.  It was shown in \cite{Moffatt:2013ty} that  for each $k\in \mathbb{N}_0$ the set 
of all ribbon graphs that have a partial dual of Euler genus at most $k$ is minor-closed, and so $X_1$--$X_3$  cannot be minors of ribbon graphs that have partial duals of Euler genus less than 1. 

For the converse we prove that if a ribbon graph does not admit a plane- or \rpt-biseparation, then it has an  $X_1$--$X_3$-minor. The result then follows from Theorem~\ref{t2}.  If $G$ is an orientable ribbon graph that does  not admit a plane-biseparation then, from \cite{MR2928906},  it  contains $X_2$ or $X_3$ as a minor. Now suppose that $G$ is a non-orientable ribbon graph that does not admit an \rpt-biseparation. To complete the proof we need to show that $G$ has an \mbox{$X_1$--$X_3$-minor.} By Equation~\eqref{e1}  it is enough to show that when $G$ is a bouquet it has   $X_1$ or $X_2$ as a minor. (To see why, suppose $G$ has no  \rpt-biseparation, and let $T$ be the edge set of a spanning tree of $G$. Then $G^T$ is a bouquet. If $G^T$ has $X_1$ or $X_2$ as a minor, by Equation~\ref{e1}, $G=(G^T)^T$ has  $X_1^T$ or $X_2^T$ as a minor. Finally, $X_1^T=X_1$, and $X_2^T = X_2$ or $X_2^T=X_3$.)
Assume that $G$ is a non-orientable bouquet that does not admit an \rpt-biseparation. 
We can write $G=G_O \sqcup G_N$ where $G_O$ is the ribbon subgraph of $G$ consisting of all of the orientable loops, and $G_N$  the non-orientable loops.  If $G_N$ has two edges that are not interlaced then these two edges induce an $X_1$-minor and we are done. Otherwise  $G_N$ consists of $q\geq 1$ non-orientable loops that all interlace each other (i.e., the edges are met in the cyclic order $1\,2\cdots q\,1\,2\cdots q$). Assume this is the case. 

Next consider $G_O$. Suppose the intersection graph $\I(G_O)$ is not bipartite. It then contains an odd cycle of length at least 3.   This odd cycle corresponds to a ribbon subgraph of $G$ that is  an orientable bouquet with $p\geq 3$ edges, for $p$ odd,  that meet the vertex in the cyclic order 
$ 2\,1\,3\,2\,4\,3\cdots p\,(p-1)\,1\,p$. Contracting the edges corresponding to $4, 5, \ldots ,p$ results in a copy of $X_2$ (see the example in \cite{Moffatt:2013ty}).   

Now suppose $\I(G_O)$ is  bipartite. Consider the graph $\I'(G)$ which is obtained from $\I(G)$ by identifying all of the negatively weighted vertices into a single negatively weighted vertex $v$, and deleting any loops created in this process (so all of the edges in $G_N$ are represented by a single vertex $v$).  
$\I'(G)$  cannot be bipartite since otherwise $G$ would admit an \rpt-biseparation (if $\I'(G)$ is 2-coloured then the set $A$ of all edges of $G$ of a  single colour defines an \rpt-biseparation). 
We therefore have that $\I'(G)$ is not bipartite and therefore contains an odd cycle. Let $C$ be a minimal odd cycle of $\I'(G)$. Since $\I(G_O)=\I'(G) \ba v$ is bipartite, this odd cycle must contain $v$. Furthermore, by minimality, the subgraph of $\I'(G)$ induced by $C$ must also be $C$ (since adding any other edge between the vertices of $C$ will create a smaller cycle).  Suppose the vertices of $C$ are $1,2,\cdots ,2m,v$ in that order.  Consider the ribbon subgraph $H$ of $G$ corresponding to $C$. Since $C$ is minimal, $H$ consists of $2m$ orientable loops, which we name $1, \ldots, 2m$, whose ends appear in the  order
$1\,2\,1\,3\,2\, 4\, 3 \cdots (2m-1) \,(2m-2) \,(2m) \,(2m-1) \,(2m)$, together with some number of non-orientable loops such that a non-orientable loop interlaces with 1, a (not necessarily distinct) orientable loop interlaces with $2m$, and all of the non-orientable loops interlace with each other (since the non-orientable part is $G_N$). See Figure~\ref{f2}.  
With respect to this order, we name the arcs on the boundary of the unique vertex of $H$ that are bounded by the orientable edges on $H$ as follows:   
%\begin{linenomath} %to work around line numbers bug
\begin{multline*}
 \varepsilon \,1 \, \alpha  \, 2\, \gamma_1 \,1 \, \delta_1\,  3\, \gamma_2 \,2\, \delta_2 \,4\, \gamma_3 3    \cdots    (2m-1)\, \gamma_{(2m-2)} \,
 \\
 (2m-2) \,  \delta_{(2m-2)}\, 
 (2m) \,\gamma_{(2m-1)}\, (2m-1) \,  \beta  \,(2m) \, \varepsilon,
 \end{multline*}
%\end{linenomath}
 as in Figure~\ref{f2}.  
 The ends of the non-orientable edges of $H$ lie on these arcs, and we will use the names of these arcs to discuss the possible ways that the non-orientable edges in $H$ can be interlaced with its orientable edges, and for each case to describe how the find the required ribbon graph minor.

\begin{figure}[t]
\centering
\labellist \small\hair 2pt
\pinlabel {$1$}   at 12 162
\pinlabel {$2$}  at 14 80
\pinlabel {$3$}   at  75 15
\pinlabel {$2m$}  [l] at  273 162
\pinlabel {$\alpha$}   at  87 166 
\pinlabel {$\beta$}   at  201 166 
\pinlabel {$\gamma_1$}  at  85 133  
\pinlabel {$\gamma_2$}   at  102 98 
\pinlabel {$\gamma_3$}   at  133 82  
\pinlabel {$\gamma_{(2m-1)}$}  [r] at  215 136 
\pinlabel {$\varepsilon$}   at  144 203
 \pinlabel {$\delta_1$}   at  49 96
  \pinlabel {$\delta_2$}   at  93 52
\endlabellist
 \includegraphics[scale=0.6]{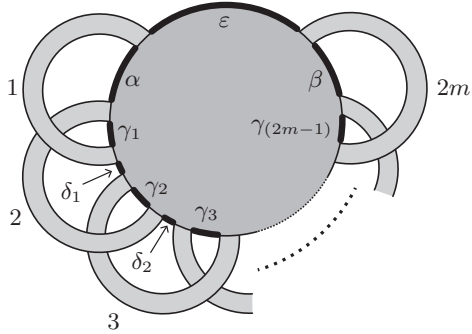}
\caption{$H$ without its non-orientable loops, and the labelled arcs on its vertex.}
\label{f2}
\end{figure}

First suppose that $m=1$. If there is a single non-orientable loop $e$ interlacing both $1$ and $2$ then the ribbon subgraph with edges $1,2,e$ is given by the order of the ends $ 1 \,e \,2 \,1 \,e \,2$ or $1 \,2 \,e \,1 \,2 \,e$   around the vertex. In either case contracting $e$ results in an $X_1$-minor. So suppose that no non-orientable loop interlaces both 1 and 2. Then there is a  non-orientable loop $e$ interlacing  $1$ but not 2, and a  non-orientable loop $f$ interlacing  $2$ but not 1. Then the ends of $e$ are therefore on $\alpha$ and $\varepsilon$, or $\gamma_1$ and $\beta$ (otherwise it interlaces 2).   Similarly, the ends of $f$ are therefore on $\beta$ and $\varepsilon$, or $\gamma_1$ and $\alpha$. Using the fact that all non-orientable loops interlace, $H$ must therefore contain one of the four ribbon subgraphs with edges $1,2,e,f$ and their positions given by one of $1\,e\,2\,1\,f\,2\,e\,f$, or $1\,f\,e\,2\,f\,1\,2\,e$, or $1\,2\,e\,1\,f\,e\,2\,f$, or $1\,f\,2\,e\,f\,1\,e\,2$. In all of these cases contracting the edges 1 and 2 results in a copy of $X_1$, as required. Thus we can obtain the required minor when $m=1$.

Now suppose that $m>1$. First consider the case where there is a single non-orientable loop $e$ interlacing both $1$ and $2m$. Then $e$ cannot have ends on either of the arcs $\gamma_1$ or $\gamma_{(2m-1)}$, as this could mean that $C$ is not a minimal odd cycle in $\I'(G)$ (given by the edges $1,2,e$ or $(2m-1),(2m),e$). Thus $e$ has ends on $\alpha$ and $\beta$. Deleting all of the non-orientable loops of $H$ except for $e$, then contracting  the edges $e, 2,3,\ldots , (2m-1)$  gives an $X_1$-minor. 

Now consider the case when  no non-orientable loop  interlaces both $1$ and $2m$.  
Then there is a  non-orientable loop $e$ interlacing  $1$ but not $2m$, and a  non-orientable loop $f$ interlacing  $2m$ but not 1, 
$e$ has an end on $\alpha$ or $\gamma_1$, and $f$ has an end on $\beta$ or $\gamma_{(2m-1)}$. 

First suppose that $e$ has an end on $\gamma_1$. Then its other end must be on $\delta_1$ or $\gamma_2$. (This is since either $e$ would not be interlaced with 1, or the edges $1,2,e$ will induce an odd cycle smaller that $C$ in $\I'(G)$.)  Consider $f$. Since all non-orientable loops interlace and $f$ does not interlace with $1$, we have that if $e$ has an end on $\delta_1$ then $f$ has an end on $\delta_1$, and if  $e$ has an end on $\gamma_2$ then  $f$ has an end on $\delta_1$ or $\gamma_2$. In all three cases the edges $1,2,e,f$ induce an odd cycle smaller that $C$ in $\I'(G)$. Thus $e$ cannot have an end on $\gamma_1$, and must have an end on $\alpha$.
A similar argument shows that $f$  cannot have an end on $\gamma_{(2m-1)}$, and  must have an end on $\beta$. We now examine where the other ends of $e$ and $f$ can lie.

If $e$ has an end on some $\gamma_i$ for $i\in \{2, \ldots ,(2m-1)\}$  (we have shown that it cannot have an end on $\gamma_1$) then the edges $e, i, i+1$ will induce an odd cycle smaller that $C$ in $\I'(G)$.
Similarly,  if $f$ has an end on some $\gamma_i$ for $i\in \{1, \ldots ,(2m-2)\}$  (we have shown that it cannot have an end on $\gamma_{(2m-1)}$) then the edges $f, i, i+1$ will induce an odd cycle smaller that $C$ in $\I'(G)$. 
 Thus $e$ has one end on  $\alpha$, and one on $\varepsilon$ or on some $\delta_i$; and $f$ has one end on  $\beta$ and one on $\varepsilon$ or on some $\delta_j$. 
If one of $e$ or $f$ has an end on  $\delta_i$, then, by interlacement, the other has an end on some $\delta_j$.
  If $e$ has an end on $\delta_{2k-1}$, for some $k$, then the edges $e,1, \ldots , 2k$ will induce an odd cycle smaller that $C$ in $\I'(G)$.
Similarly, if $f$ has an end on $\delta_{2k}$, for some $k$, then the edges $2k+1, \ldots , 2m$ will induce an odd cycle smaller that $C$ in $\I'(G)$.
 It follows that if $e$ has an end on $\delta_{2k}$ then  $f$ has an end on $\delta_{2j-1}$ for some $j\leq k$, but then the edges $e,f, 2j, \ldots , (2k+1)$ will induce an odd cycle smaller that $C$ in $\I'(G)$. 
 Thus we have that  $e$ has  ends on  $\alpha$ and $\varepsilon$ and  $f$ has  ends on  $\beta$ and $\varepsilon$, with $e$ and $f$ interlaced. Deleting all of the non-orientable loops of $H$ except for $e$ and $f$, and contracting all of the orientable loops results in a copy of $X_1$.  Thus we have shown in all cases that $H$, and therefore $G$ has an $X_1$-minor. This completes the proof of the theorem.
\end{proof}

We conclude with an application of Theorem~\ref{t1} to knot theory.  Dasbach,  Futer, Kalfagianni, Lin and Stoltzfus, in  \cite{MR2389605}, described how  every alternating link can be represented by a ribbon graph. This construction readily extends to link diagrams on other surfaces. Using, from \cite{SM}, that a ribbon graph represents a checkerboard colourable  diagram of a link in real projective space, $\mathbb{R}\mathrm{P}^3$, if and only if it has a plane or \rp partial dual immediately gives the following corollary.
\begin{corollary}
A ribbon graph  $G$ represents a checkerboard colourable  diagram of a link in real projective space  if and only if it has no minor equivalent to $X_1$, $X_2$, or $X_3$. 
\end{corollary}
Finding the corresponding result for non-checkerboard colourable diagrams is an interesting  open problem.

% ------------------------------------------------------------------------

%\subsection*{Acknowledgment}

\bibliographystyle{plain}
\bibliography{references}

\begin{thebibliography}{1}

\bibitem{MR1906909}
B{{\'e}}la Bollob{{\'a}}s and Oliver Riordan.
\newblock A polynomial of graphs on surfaces.
\newblock {\em Math. Ann.}, 323(1):81--96, 2002.

\bibitem{MR2507944}
Sergei Chmutov.
\newblock Generalized duality for graphs on surfaces and the signed
  {B}ollob{\'a}s-{R}iordan polynomial.
\newblock {\em J. Combin. Theory Ser. B}, 99(3):617--638, 2009.

\bibitem{MR2389605}
Oliver~T. Dasbach, David Futer, Efstratia Kalfagianni, Xiao-Song Lin, and
  Neal~W. Stoltzfus.
\newblock The {J}ones polynomial and graphs on surfaces.
\newblock {\em J. Combin. Theory Ser. B}, 98(2):384--399, 2008.

\bibitem{MR3086663}
Joanna~A. Ellis-Monaghan and Iain Moffatt.
\newblock {\em Graphs on surfaces}.
\newblock Springer Briefs in Mathematics. Springer, New York, 2013.
\newblock Dualities, polynomials, and knots.

\bibitem{MR2928906}
Iain Moffatt.
\newblock Partial duals of plane graphs, separability and the graphs of knots.
\newblock {\em Algebr. Geom. Topol.}, 12(2):1099--1136, 2012.

\bibitem{MR2994404}
Iain Moffatt.
\newblock Separability and the genus of a partial dual.
\newblock {\em European J. Combin.}, 34(2):355--378, 2013.

\bibitem{Moffatt:2013ty}
Iain Moffatt.
\newblock Excluded minors and the ribbon graphs of knots.
\newblock {\em J. Graph Theory}, to appear.

\bibitem{SM}
Iain Moffatt and Johanna Str\"{o}mberg.
\newblock On the ribbon graphs of links in real projective space.
\newblock {\em preprint}.

\bibitem{MR2099147}
Neil Robertson and P.~D. Seymour.
\newblock Graph minors. {XX}. {W}agner's conjecture.
\newblock {\em J. Combin. Theory Ser. B}, 92(2):325--357, 2004.

\end{thebibliography}

% ------------------------------------------------------------------------
\end{document}